\documentclass{amsart}
\usepackage{amsmath}
\usepackage{amssymb}
\usepackage{graphicx}
\input xy
\xyoption{all}
\newtheorem{theorem}{Theorem}[section]
\newtheorem{conjecture}{Conjecture}
\newtheorem{lemma}[theorem]{Lemma}
\newtheorem{corollary}[theorem]{Corollary}

\newtheorem{example}[theorem]{Example}

\theoremstyle{definition}

\newcommand{\ann}{\mbox{\rm ann}}
\newcommand{\End}{\mbox{\rm End}}

\newcommand{\degmin}{\mbox{\rm degmin}}

\newcommand{\N}{\mathbb N}
\newcommand{\Z}{\mathbb{Z}}

\begin{document}
    \title[left algebraic subgroups]{A note On subgroups in a division ring that are left algebraic over a division subring}
    \author[B. X. Hai]{Bui Xuan Hai}\thanks{The first and the third authors were funded by  Vietnam National University HoChiMinh City (VNUHCM) under grant number C2018-18-03.}
    \author[V. M. Trang]{Vu Mai Trang} \author[M. H. Bien]{Mai Hoang Bien}
   \email{bxhai@hcmus.edu.vn;trangvm8234@gmail.com; mhbien@hcmus.edu.vn}
\address{Faculty of Mathematics and Computer Science, VNUHCM-University of Science, 227 Nguyen Van Cu Str., Dist. 5, Ho Chi Minh City, Vietnam.}


 \maketitle

 \begin{abstract}  Let $D$ be a division ring with center $F$ and $K$ a division subring of $D$. In this paper, we show that a non-central normal subgroup $N$ of the multiplicative group $D^*$ is left algebraic over $K$ if and only if so is $D$ provided $F$ is uncountable and contained in $K$. Also, if $K$ is a field and the $n$-th derived subgroup $D^{(n)}$ of $D^{*}$ is left algebraic of bounded degree $d$ over $K$, then $\dim_FD\le d^2$.
\end{abstract}
\section{Introduction}        

Let $D$ be a division ring with center $F$ and $K$ a division subring of $D$. Recall that an element $a\in D$ is \textit{left algebraic} over $K$ if there exist $a_0,a_1,\cdots,a_n\in K$ not all are zeros such that $a_0+a_1a+\cdots+a_na^n=0$, or equivalently, there exists a non-zero polynomial $f(t)\in K[t]$ whose coefficients are written on the left such that $f(a)=0$. Here, we have to emphasize the convention that for polynomials over a  ring $R$ (commutative or not), a polynomial $f(t)\in R[t]$ can be written in two ways such as
$$f(t)=\sum_{\text{finite}}a_it^i=\sum_{\text{finite}}t^ia_i.$$
However, if $a\in S$, where $S$ is a ring containing $R$, the substitution functions may give the different values, i.e., we may have
$$\sum_{\text{finite}}a_ia^i\neq\sum_{\text{finite}}a^ia_i.$$ 
In this paper, we always mean $$f(a)=\sum_{\text{finite}}a_ia^i,$$
and then $a$ is called a \textit{right root} of $f(t)$. A \textit{left root} of $f(t)$ and a \textit{right algebraic element} are defined similarly. A subset $S$ is called \textit{left algebraic} (resp., \textit{right algebraic}) over $K$ if every element in $S$ is left algebraic (resp., right algebraic) over $K$. If $K$ is central, that is, $K\subseteq F$, then the  right algebraicity coincides with the left one.  However, if $K$ is not central, then there are division rings $K\subseteq D$ such that $D$ is left algebraic but not right algebraic over $K$. Observe that with division rings $K\subseteq D$ which were presented by Cohn in \cite{Pa_Co_61}, we can show that the division ring $D$ is left algebraic but not right algebraic over $K$. In Section~\ref{s2}, we give a natural and simple example of division rings $K\subseteq D$ such that $D^*$ contains a proper normal subgroup $N$ which is left algebraic but not right algebraic over $K$. In the case when $K$ is a subfield of $D$, it is not known whether $D$ is left algebraic over $K$ provided it is  right algebraic over $K$ (see \cite[Page 1610]{Bell}).

The notion of one-sided (right or left) algebraicity has been introduced to study polynomials over division rings. For example, \cite{Bo_Ja_64} (see Chapter 7) is one of the oldest books mentioning this notion. Some special cases of one-sided algebraicity were studied by several authors. For instance, C. Faith \cite{Pa_Fa_60} investigated division rings that are radical over their division subrings, P. M. Cohn and A. H. Schofield \cite{Pa_Co_61,Pa_Sc_85} studied the left and right dimensions of a division ring $D$ over its division subring $K$. Recently, division rings whose elements are algebraic (left or right) over their division subrings have been received a considerable attention (see, for example \cite{Pa_AaAkBi_17,Pa_BeRo_14,Bell,Pa_DeBiHa_18, Pa_Ha_02} and references therein). 

In this note, firstly, we are interested in the question whether the algebraicity of some subset in a division ring $D$ over its proper division subring $K$ may entail the algebraicity of whole $D$. More exactly, we pose the following conjecture to study.
\begin{conjecture}
	Let $D$ be a division ring with center $F$, $K$ a division subring of $D$ containing $F$, and $N$ a subnormal subgroup of $D^*$. If $N$ is non-central, then $N$ is left algebraic (resp., right algebraic) over $K$ if and only if so is $D$.
\end{conjecture}\label{conj:1}
In Section~\ref{s3} (see Theorem~\ref{t2.3}), we give the answer to this conjecture in the case when $N$ is a normal subgroup of $D^*$ and $F$ is uncountable. 

Secondly, in Section~\ref{s4}, we study a division ring $D$ whose $n$-th derived subgroup $D^{(n)}$ of $D^*$ is left algebraic of bounded degree over some subfield  (recall that the $n$-th derived subgroup $D^{(n)}$ of $D^*$ is defined as follow: $D^{(1)}=D'$ is the commutator subgroup of $D^*$ and $D^{(n)}$ is the commutator subgroup of $D^{(n-1)}$ for $n>1$). 
Note that in \cite{Bell}, it was proved that if $D$ is a division ring with center $F$ and there exists a subfield $K$ of $D$ such that $D$ is left algebraic over $K$ of bounded degree $d$, then $\dim_FD\le d^2$. This result was extended for the commutator subgroup $D'$ instead of $D$ in \cite[Theorem 17]{Pa_AaAkBi_17}. The result we get in Theorem~\ref{t3.3} generalizes this fact by considering $D^{(n)}$ for an arbitrary $n\ge 1$ instead of $D'$.

\section{Example}\label{s2} In this section, we give an example of division rings $K\subseteq D$ such that $D^*$ contains a proper normal subgroup $N$ which is left algebraic but not right algebraic over $K$.

Let $F$ be a field with an endomorphism $\sigma$ and $t$ an indeterminate. We denote by $F((t,\sigma))=\{\sum\limits_{i=n}^\infty a_it^i\mid a_i\in F, n\in \Z \}$ the ring of Laurent skew series in $t$ over $F$ with respect to $\sigma$ in which the addition is defined as usual, and the multiplication is an extension of the rule $ta=\sigma (a)t$. In general, $F((t,\sigma))$ is not a division ring (for example, if $\sigma$ is not injective, then there exists $a\in F^*$ such that $\sigma(a)=0$, so $ta=\sigma(a)t=0$). However, if $\sigma$ is injective, then $F((t,\sigma))$ is a division ring \cite[Example 1.7]{Bo_La_99}. For $\alpha=\sum_{i=n}^\infty a_it^i$, the lowest power appearing in $\alpha$ is denoted by $\degmin(\alpha)$. That is, $\degmin(\alpha)=\min\{i\mid a_i\ne 0\}.$

Now let $k$ be a field of characteristic $0$ and $\{x_0, x_1,\cdots\}$ a countable set of commuting indeterminates. Consider the field of fractions  $F=k(x_0,x_1,\dots)$ of the polynomial ring $k[x_0,x_1,\dots]$, and the endomorphism $\sigma : F\to F$  defined by $\sigma(x_{i})=x_{i+1}$ for $i\in\N$. Then, we have the following easy lemma.

\begin{lemma}\label{ml2.1} Let $F=k(x_0,x_1,\dots)$ and $\sigma : F\to F$ be defined above. Then $D=F((t,\sigma))$ is a division ring and $$\degmin(\alpha.\beta)=\degmin(\alpha)+\degmin(\beta)$$  for every $\alpha,\beta\in D$. In particular,  $\degmin(\alpha^{-1})=-\degmin(\alpha)$ for  every $\alpha\in D^*$.
\end{lemma}
\begin{proof}
	The proof of the first conclusion is essentially due to that of \cite[Example 1.7 and Proposition 14.2]{Bo_La_99}. The proof of the second one is elementary by the fact that $\sigma$ is injective.
\end{proof}
 
Now,  we are ready to give an example we have mentioned in the begining of this section.
\begin{example}{\rm Let $D=F((t,\sigma))$ be as in Lemma~\ref{ml2.1} and consider the following subset in  $D$:
$$K=F((t^2,\sigma))=\Big\{ \sum\limits_{i=n}^\infty a_{i}t^{2i}\mid n\in \Z, a_i\in F \Big\}.$$
It is easy to see that $K$ is a division subring of $D$. It is obvious that if $\alpha=\sum\limits_{i=n}^\infty a_it^i\in D$, then $$\alpha=\sum\limits_{i =2j\ge n} a_it^i+\sum\limits_{i=2j+1\ge n} a_it^i,$$ that is, $\alpha =\alpha_1+\alpha_2 t$, where $\alpha_1,\alpha_2\in K$. Hence, $\{1,t\}$ is a basis of the left vector space $D$ over $K$, which implies that the dimension  of the left vector space $D$ over $K$ is $2$. Hence, every element of $D$ is left algebraic of degree $\le 2$ over $K$.
Now, let $N=\{\alpha\in D^*\mid \degmin(\alpha)=0\}$. We claim that $N$ is a proper normal subgroup of $D^*$. 
Indeed, it is trivial that $N\ne D^*$. For $\alpha,\beta\in N$, the condition  $\degmin(\alpha)=\degmin(\beta)=0$ implies $\degmin(\alpha\beta)=\degmin(\alpha)+\degmin(\beta)=0$ and $\degmin(\alpha^{-1})=-\degmin(\alpha)=0$. Therefore, $\alpha\beta, \alpha^{-1}\in N$, which shows that $N$ is subgroup of $D^*$. Assume that $\alpha\in N$ and $\beta\in D^*$. Then,

$\degmin(\beta^{-1}\alpha\beta)=\degmin(\beta^{-1})+\degmin(\alpha)+\degmin(\beta)$

\hspace*{2.51cm}$=-\degmin(\beta)+\degmin(\beta)=0.$

As a corollary, $\beta^{-1}\alpha\beta\in N$, so $N$ is normal in $D^*$.
		
Consider the element $x_0+t\in N$. To finish example, we will show that $x_0+t$ is not right algebraic over $K$. Suppose that there exist $h_0(t^2), h_1(t^2),\cdots, h_n(t^2)\in K$ such that $h_n(t^2)\ne 0$ and $$h_0(t^2)+(x_0+t)h_1(t^2)+\cdots + (x_0+t)^n h_n(t^2)=0.$$ We seek a contradiction. Indeed, observe that, after expanding, $(x_0+t)^i$ is written as a linear sum of term $x_0^{m_0}x_{1}^{m_1}\cdots x_{i-1}^{m_{i-1}}t^{m}$ over $\Z$, so one sees that the term  $x_0^{n-1}t$ appears in $(x_0+t)^n$ but does not in $(x_0+t)^i$ with $i<n$. Moreover, all powers of $t$ appearing in $h_0(t^2),\dots, h_{n-1}(t^2)$ is even, so $x_0^{n-1}th_n(t^2)=0$, equivalently, $h_n(t^2)=0$, a contradiction. Thus, $x_0+t$ is not right algebraic over $K$. }
\end{example}

\section{Left algebraic normal subgroups in a division ring}\label{s3}

For a division ring $D$ and its division subring $K$,  ${}_KD$ and $D_K$ denote the left  and right vector space over $K$ respectively. In this section, we give the affirmative answer to Conjecture \ref{conj:1} in the case when $F$ is uncountable 
and $N$ is a non-central normal subgroup of $D^*$. 

We need some lemmas.

\begin{lemma}\label{l2.1}{\rm\cite[Page 440]{scott}}
	Let $D$ be a division ring with center $F$ and $N$ a normal subgroup of $D^*$. If $N$ is non-central, then $C_D(N)=F$.
\end{lemma}

\begin{lemma} \label{l2.2}
	Let $D$ be a division ring with center $F, K$ a division subring of $D$ containing $F$ and $a$ an element of $D$. Assume that $\alpha_1,\alpha_2,\cdots,\alpha_n$ are distinct elements in $F$  such that all elements $a-\alpha_i$ are non-zeros. Then, either $a$ is left (resp. right) algebraic over $K$ or the set  $\{(a-\alpha_i)^{-1}\mid i=1,2,\cdots, n\}$ is left (resp. right) linearly independent over $K$.
\end{lemma}
\begin{proof} It is enough to prove the lemma for the case of left algebraicity since the case of right algebraicity  is similar. Assume that $a$ is not left algebraic over $K$ and $$\beta_1(a-\alpha_1)^{-1}+\beta_2(a-\alpha_2)^{-1}+\cdots \beta_n(a-\alpha_n)^{-1}=0\eqno(1)$$ for some $\beta_i\in K$. Consider the polynomials $$f(t)=(t-\alpha_1)(t-\alpha_2)\cdots (t-\alpha_n)\in F[t]\subseteq K[t]$$ and $f_i(t)=f(t)/(t-\alpha_i)$ for $1\leq i\leq n$. Multiplying both sides of (1) on the right by $f(a)$, we get $\beta_1 f_1(a)+\beta_2 f_2(a)+\cdots+\beta_n f_n(a)=0$. This shows that $a$ is a right root of the polynomial $g(t)=\beta_1 f_1(t)+\beta_2 f_2(t)+\cdots+\beta_n f_n(t)$, so $g(t)\equiv 0$ because $a$ is not left algebraic over $K$. Then, for every $1\le i\le n$, we have $0=g(\alpha_i)=\beta_if_i(\alpha_i)$. Therefore, $\beta_i=0$ for all $i$. Hence, the set $\{(a-\alpha_i)^{-1}\mid i=1,2,\cdots, n\}$ is left linearly independent over $K$. 
\end{proof}
Note that the special case of Lemma \ref{l2.2} when $K=F$ was considered in \cite[Proposition 5.2.21]{Bo_Ro_91}.

\begin{theorem}\label{t2.3} Let $D$ be a division ring with uncountable center $F$, $K$ a division subring of $D$ containing $F$ and $N$ a normal subgroup of $D^*$. If $N$ is non-central, then $N$ is left algebraic (resp., right algebraic) over $K$ if and only if so is $D$.
\end{theorem}
\begin{proof} We show that the theorem is true for the left case since the proof for the right case is similar. Thus, assume that $N$ is a non-central normal subgroup of $D^*$ which is left algebraic over $K$. For any $a\in D$, we have to prove that $a$ is left algebraic over $K$. If $a\in C_D(N)$, then by Lemma~\ref{l2.1}, $a\in F\subseteq K$, and there is nothing to prove. Now, assume that $a\not\in C_D(N)$. Take $b\in N$ such that $d=ba-ab\ne 0$. For every $\alpha\in F$, one has $$d=ba-ab=b(a+\alpha)-(a+\alpha)b$$$$=b(a+\alpha)(1+(a+\alpha)^{-1}b^{-1}(a+\alpha)b)=b(a+\alpha)(1+c),$$ where $c=(a+\alpha)^{-1}b^{-1}(a+\alpha)b\in N$. Since $c$ is left algebraic over $K$ and $c+1\ne 0$, the element $(c+1)^{-1}$ is left algebraic over $K$. Consequently, $d^{-1}b(a+\alpha)=(c+1)^{-1}$ is left algebraic over $K$. Hence, there exist $\beta_1,\beta_2,\cdots, \beta_n\in K$ such that $$\beta_n (d^{-1}b(a+\alpha))^n+\beta_{n-1} (d^{-1}b(a+\alpha))^{n-1}+\cdots+\beta_1 (d^{-1}b(a+\alpha))+1=0.$$ We can write this equality as follows: $$1=(-\beta_n (d^{-1}b(a+\alpha))^{n-1}-\cdots-\beta_1 (d^{-1}b))(a+\alpha).$$ 
Therefore, $(a+\alpha)^{-1}=-\beta_n (d^{-1}b(a+\alpha))^{n-1}-\cdots-\beta_1 d^{-1}b$ is in the left vector $K$-subspace $W$, which is generated by the subgroup $\langle a,b,d\rangle$  of $D^*$ generated by $a,b,d$. Since $\langle a,b,d\rangle$ is a finitely generated subgroup, the cardinality of the basis of $W$ over $K$ is countable.
Observe that $F$ is uncountable, so is the set $$\{(a+\alpha)^{-1}\mid \alpha \in F \}.$$ As a corollary, the set $\{(a+\alpha)^{-1}\mid \alpha \in F \}$ is left linearly dependent over $K$. In view of Lemma~\ref{l2.2}, $a$ is left algebraic over $K$.  
\end{proof}
The following corollary gives the affirmative answer to \cite[Problem 13]{Pa_Ma_00} in the case of uncountable center $F$ of $D$ and $N$ is normal in $D^*$. 

\begin{corollary}\label{c2.4} Let $D$ be a division ring with uncountable $F$. Assume that $N$ is a non-central normal subgroup of $D^*$. If $N$ is algebraic over $F$, then so is $D$. 
\end{corollary}

\section{Left algebraic $n$-th derived subgroup of bounded degree}\label{s4}
Let $D$ be a division ring and $K$ a subfield of $D$. In this section, we prove that if for some integer $n\geq 1$, the $n$-th derived subgroup $D^{(n)}$ is left algebraic of bounded degree $d$ over $K$, then $\dim_FD\le d^2$

The proof of the following lemma is elementary, so we omit it. 
\begin{lemma}\label{l3.0}
	Let $D$ be a division ring, $K$ a division subring and $x$ an element in $D$. The following conditions are equivalent.
	\begin{enumerate}
		\item The element $x$ is left algebraic (resp., right algebraic) over $K$ of degree $d$. 
		\item $d$ is the largest integer such that $\{x^n\mid n=1,2,\cdots, d-1\}$ is a left (resp., right) linearly independent set in ${}_KD$ (resp., $D_K$). 
		\item $d$ is the largest integer such that the sum $\sum_{i=0}^{d-1} Kx^i$ (resp., $\sum_{i=0}^{d-1} x^iK$) is a direct sum in ${}_KD$ (resp., $D_K$). $\square$
	\end{enumerate} 
\end{lemma}

\begin{lemma}\label{l3.1}
	Let $D$ be a division ring with infinite center, $N$ a non-central normal subgroup of $D^*$ and $K$ a subfield of $D$. If $N$ is left (or right) algebraic over $K$ of bounded degree, then $D$ is centrally finite, that is, $D$ is a finite dimensional vector space over its center.
\end{lemma}
\begin{proof} Since $N$ is non-central, there exists $a\in N\backslash F$. For every $x\in D^*$, one has $axa^{-1}x^{-1} = a(xa^{-1}x^{-1})\in N$, so $axa^{-1}x^{-1}$ is left (resp., right) algebraic of bounded degree over $K$. By \cite[Theorem 11]{Pa_AaAkBi_17}, $D$ is centrally finite.
\end{proof}

\begin{corollary}\label{c3.2}
	Let $D$ be a division ring with infinite center. For any positive integer $n$, if the $n$-the derived subgroup $D^{(n)}$ is left (or right) algebraic of bounded degree over some subfield of $D$, then $D$ is centrally finite.
\end{corollary}
\begin{proof}
	If $D^{(n)}$ is central, then $D^*$ is solvable, so $D$ is a field by \cite[14.4.4, Page 440]{scott}. Hence, we can assume that $D^{(n)}$ is non-central. Since $D^{(n)}$ is normal in $D^*$, by Lemma~\ref{l3.1}, $D$ is centrally finite.
\end{proof}

The following theorem extends \cite[Theorem 1.3]{Bell} and \cite[Theorem 17]{Pa_AaAkBi_17}.
\begin{theorem}\label{t3.3}
	Let $D$ be a division ring with infinite center $F, K$ a subfield of $D$ and  $n$ be a positive integer. If the $n$-th derived subgroup $D^{(n)}$ is left (or right) algebraic of bounded degree $d$ over $K$, then $\dim_FD\le d^2$.
\end{theorem}
\begin{proof} We prove the theorem for the left case since the right case is similar. Without loss of generality, we assume that $K$ is a maximal subfield of $D$. According to Corollary~\ref{c3.2}, $\dim_FD=m^2<\infty$. We must show that $m\le d$. By \cite[Theorem 7]{Pa_AaBi_18}, there exists $x\in D^{(n)}$ such that $L=F(x)$ is a maximal subfield of $D$. It is well known that $\dim_FL=m$ (or see \cite[Proposition 15.7 and Theorem 15.8]{Bo_La_99}). One has that $D$ is a left $D\otimes _FL$-module in which the operator is defined by $(\alpha\otimes x^i) \beta=\alpha\beta x^i$ for every $\alpha,\beta\in D$ and $i\in \N$. Observe that $D\otimes _FL$ is simple, so $D$ is faithful. On the other side, $D$ may be considered as a left $K$-space. Now, consider  $T\in \End_KD$ which is defined by $T(\alpha)=\alpha x$ for every $\alpha \in D$.  We claim that the set $\{T^{i}\mid i=0,1,\cdots,m-1 \}$ is left linearly independent over $K$. Indeed, assume that $\sum\limits_{i=0}^{m-1}c_iT^i=0$ for some $c_0,c_1,\cdots,c_{m-1}\in K$. Then, for every $\alpha \in D$, $$0=\Big(\sum\limits_{i=0}^{m-1}c_iT^i\Big) (\alpha) =\sum\limits_{i=0}^{m-1}c_i\alpha x^i=\Big(\sum\limits_{i=0}^{m-1}c_i\otimes x^i\Big) \alpha.$$ Observe that $D$ is faithful, so $\sum\limits_{i=0}^{m-1}c_i\otimes x^i=0$, which implies that $$c_0=c_1=\cdots=c_{m-1}=0.$$ The claim is proved. The next claim is that there exists $y\in D$ such that $Ky+KT(y)+\cdots+KT^{m-1}(y)$ is a direct sum. Indeed,	let $t$ be an indeterminate and $K[t]$ be the polynomial ring in $t$ over $K$. Then, we can consider $D$ as left $K[t]$-module with operator defined by the rule $f(t).\alpha=f(T)(\alpha)$ for every $\alpha\in D$ and $f(t)\in K[t]$. Since $\dim_KD<\infty$, there exists a non-zero element $g(t)\in K[t]$ such that $g(T)=0$. Hence, for every $\alpha\in D$, one has $g(t)\alpha=g(T)(\alpha)=0$, so it follows that $D$ is torsion as left $K[t]$-module. Moreover, it is obvious that $D$ is finitely generated as left $K$-space, so is $D$ as a left $K[t]$-module. Therefore, $D$ is torsion finitely generated as a left module over  a PID. Hence, there exist $f_1(t),f_2(t),\cdots,f_\ell(t)\in K[t]$ such that $$\langle f_1(t)\rangle\supseteq \langle f_2(t)\rangle \supseteq \cdots \supseteq \langle f_\ell(t)\rangle$$ and an isomorphism $$\phi : K[t]/\langle f_1(t)\rangle\oplus K[t]/\langle f_2(t)\rangle\oplus \cdots \oplus K[t]/\langle f_\ell(t)\rangle \to D,$$ where $\langle f(t)\rangle$  denotes the ideal of $K[t]$ generated by some element $f(t)\in K[t]$. Put $y=\phi(1+\langle f_\ell(t)\rangle)\in D$. We will show that $y$ is the element we need to find. Indeed, assume that 	$f(t)=c_0+c_1t+\cdots+c_{m-1}t^{m-1}\in K[t]$ such that $f(T) (y)=0$. Then, $f(t)y=0$, equivalently, $f(t)\in \ann_{K[t]}y$.	Observe that
	 $\langle y\rangle\cong R/\langle f_\ell(t)\rangle$ and by direct calculation, one has $$\ann_{K[t]}(D)=\bigcap\limits_{i=1}^\ell \ann_{K[t]} K[t]/\langle f_i(t)\rangle=\bigcap\limits_{i=1}^\ell \langle f_i(t)\rangle=\ann_{K[t]}y.$$ Hence, $f(t)\in \ann_{K[t]}y=\ann_{K[t]}D$. As a corollary, $f(T)(\alpha)=0$ for every $\alpha\in D$ which contradicts to the fact that $\{T^{i}\mid i=0,1,\cdots,m-1\}$ is left linearly independent over $K$. Therefore, the claim is proved. Put $u=yxy^{-1}$. Then, $$K+Ku+\cdots+Ku^{m-1}=(Ky+Kyx+\cdots+Kyx^{m-1})y^{-1}$$ is a direct sum. By Lemma~\ref{l3.0}, it follows that $u$ is left algebraic of degree $m$ over $K$. On the other hand, since $x$ is in $D^{(n)}$ and left algebraic of bounded degree $d$ over $K$, so is $u\in D^{(n)}$. Thus, again by Lemma~\ref{l3.0}, $m\le d$. The proof is now complete.
\end{proof}

\noindent

\textbf{Acknowledgements}

\bigskip

The authors would like to express their sincere gratitude to the Editor and also to the referee for his/her careful reading and comments.

\end{document}